\documentclass[12pt]{article}

\usepackage{amsmath, amscd, amssymb, amsthm, xspace, graphicx}
\usepackage[all]{xy}

\theoremstyle{plain}
\newtheorem{theorem}{Theorem}[section]
\newtheorem{lemma}[theorem]{Lemma}

\newcommand{\C}{\mathbb{C}}

\newcommand{\Q}{\mathbb{Q}}
\newcommand{\Z}{\mathbb{Z}} 
\newcommand{\N}{\mathbb{N}}

\title{Roots of Modular Units}

\author{A. Beeson}

\begin{document}
\maketitle

\begin{abstract}
Let $p$ be a prime. We prove that if a modular unit has a $p^{th}$ root that is again a modular unit then the level of that root is at most $p$ times the level of the original unit. 
\end{abstract}

\section{Introduction}

We prove the result in section~\ref{level}, but because the literature contains inconsistent definitions, in section~\ref{background} we give  a summary of the relevant definitions. The reader interested in learning more about the theory of modular functions should see~\cite{diamondshurman:2005}.    

Our main results is that the $p^{th}$ root of a modular unit, if it is again a modular unit, has the level that one would expect.  The source of this question was the study of the Siegel functions and their square roots. See~\cite{kubert:1981}. The result could likely be extended using similar techniques to the general case of a cyclic Galois group.

\section{On the level of a root of a modular function}\label{level}

\begin{theorem}  Let $\mathcal{F}_N$ be the field of modular functions of level $N$ with Fourier coefficients in $\Q(\zeta_N)$; that is, $\mathcal{F}_N$ is the fixed field of $\Gamma(N)$ inside $\mathcal{F}$, the full field of modular functions.  If $p$ is a rational prime and $f(z)\in \mathcal{F}_{N}\setminus \mathcal{F}_{N}\sp{p}$ is a modular unit and $f(z)^{1/p}$ has level $M$ for some $M\in\N$ with $N | M$ then, in fact, $f(z)^{1/p}$ has level $pN$.  
\end{theorem}
\begin{proof}
We assume $f(z)^{1/p}$ is known to be invariant under the subgroup $\Gamma(M)\subseteq \Gamma(N)$.  We will show that $f(z)^{1/p}$ is invariant under $\Gamma(pN)$.

Let $\Gamma_{1}$ be the subgroup of $\Gamma(N)$ that fixes $f(z)^{1/p}$; ie, for all $A \in \Gamma_{1}$ and $z \in \mathcal{H}^{*}$,  

\[
f(A\circ z)^{1/p} = f(z)^{1/p}.
\]

\noindent Because $\mathcal{F}_{N}(f(z)^{1/p})$ is a degree $p$ extension of $\mathcal{F}_{N}$, the index $[\Gamma(N):\Gamma_{1}] = p$, and, thus, $\Gamma_{1}$ is a finite index subgroup of $\Gamma$ as well.  Furthermore, because $f(z)^{1/p}$ is of level $M$, $\Gamma(M) \subseteq \Gamma_{1}$.  So we have the linear ordering of fields 
\[
\Q(j(z)) \subseteq \mathcal{F}_N \subseteq \mathcal{F}^{\Gamma_1} \subseteq \mathcal{F}_M \subseteq \mathcal{F}
\]
 where $\mathcal{F}^{\Gamma_1}$ denotes the fixed field of $\Gamma_1$ inside $\mathcal{F}$.


Let $\mathcal{D}$ be a fundamental domain for $\Gamma$; so $\mathcal{D}$ is a simply connected subset of $\mathcal{H}^{*}$ such that $\mathcal{D}$ contains precisely one point from each $\Gamma$-orbit.  If $\mathcal{D}_{1}$ is a fundamental domain for the subgroup of finite index $\Gamma_{1}\subseteq \Gamma$ then it is made up of translates of $\mathcal{D}$ by a full set of coset representatives for $\Gamma_{1}$ inside $\Gamma$.  Such a translate is called a \emph{modular triangle}.  Define the \emph{fan width} of a fundamental domain at a cusp $\alpha$ to be the order of the cyclic group that permutes the  $\Gamma_{1}$-inequivalent modular triangles meeting at $\alpha$.  Schoeneberg proves in~\cite{schoeneberg:1974} that the conductor of a group $\Gamma_{1}$ is equal to the least common multiple of the fan widths at the rational cusps.

The index of $\Gamma(N)$ in $\Gamma$ is the size of $SL_{2}(\Z/N\Z)/\{\pm I\}$, which is
\[ 
l = [\Gamma:\Gamma(N)] = \frac{1}{2}N^{3}\prod_{p | N}(1 - \frac{1}{p^{2}}).
\] 
\noindent As observed above, $[\Gamma(N):\Gamma_{1}] = p$, hence the group of automorphisms of $\mathcal{F}^{\Gamma_{1}}$ fixing $\mathcal{F}_{N}$ is cyclic.  If we let $\sigma$ be a generator then $\Gamma(N)$ decomposes as the disjoint union
\[
\Gamma(N)=\Gamma_{1}\cup\sigma\Gamma_{1}\cup\sigma^2\Gamma_1\cup\cdots\cup\sigma^{p-1}\Gamma_{1}.
\]
\noindent And if $\{A_{1}, \ldots, A_{l}\}$ is a complete set of representative for $\Gamma/\Gamma(N)$ then  
\[
\{A_{1}, \ldots, A_{l}, \sigma A_{1}, \ldots, \sigma A_{l}, \sigma^2 A_{1}, \ldots, \sigma^2 A_{l}, \ldots, \sigma^{p-1} A_1, \ldots, \sigma^{p-1}A_l \}
\]
 is a complete set of representatives for $\Gamma/\Gamma_{1}$.  Recalling our notation $C(\Gamma(N))$ for the cusps of $\mathcal{H}/\Gamma(N)$, we see that 
\[
C(\Gamma(N)) \subset C(\Gamma_{1})
\]
\noindent and if $\alpha$ is a cusp of $\Gamma_{1}$ then $\sigma^i \alpha \in C(\Gamma(N))$ for some $i$ with $1 \leq i \leq p$. 

Choose $\{A_{1}, \ldots, A_{l}\}$ to be a complete set of representatives for $\Gamma/\Gamma(N)$ such that 
\[
\mathcal{D}_{N} =  \cup_{i}A_{i}(\mathcal{D})
\]
\noindent is a fundamental domain for $\mathcal{H}/\Gamma(N)$. Let $\mathcal{D}_{1}$ be a fundamental domain for $\mathcal{H}/\Gamma_{1}$.

Schoeneberg's theorem implies that the least common multiple of the fan widths for $\mathcal{D}_{N}$ is $N$. We will use this to show that, for any cusp $\alpha$ of $\Gamma_{1}$, the fan width of $\mathcal{D}_{1}$ at $\alpha$ divides $pN$ so, by Schoeneberg's theorem, the conductor of $\Gamma_{1}$ divides $pN$.  Because $f(z)^{1/p}\notin \mathcal{F}_{N}$, $f(z)^{1/p}$ must have level $pN$.

Let $\alpha$ be a cusp of $\Gamma_{1}$.  As observed above, this implies $\sigma^i \alpha$ is a cusp of $\Gamma(N)$ for at least one $i$ with $1 \leq i \leq p$.  Letting $\beta = \sigma^i \alpha$ be a translate of $\alpha$ that is a cusp of $\mathcal{D}_{N}$, we see that multiplication by $\sigma^i$ is a homeomorphism between a neighborhood of $\alpha$ and a neighborhood of $\beta$.  Thus, it suffices to prove the result for the cusps of $\mathcal{D}_{1}$ that are also cusps of $\mathcal{D}_{N}$.

\begin{lemma}  If $\alpha$ is a cusp of $\Gamma_{1}$ and of $\Gamma(N)$ that is of fan width $n$ for $\mathcal{D}_{N}$ then its width for $\mathcal{D}_{1}$ is $n$ or $pn$.   
\end{lemma}
\begin{proof}
Recall $\Gamma_{1} \subseteq \Gamma(N)$ so $\Gamma \setminus \Gamma(N) \subseteq \Gamma \setminus \Gamma_{1}$ and that the fan width $n$ of $\mathcal{D}_{1}$ at the cusp $\alpha$ is the order of the cyclic group that permutes the $n$ $\Gamma_{1}$-inequivalent triangles meeting at $\alpha$.  

If two triangles are $\Gamma_{N}$-inequivalent then they are $\Gamma_{1}$-inequivalent so, assuming the width for $\mathcal{D}_{N}$ is $n$, the width for $\Gamma_{1}$ is at least $n$.  Then since $[\Gamma(N):\Gamma_{1}] = p$, we see that the width of a triangle for $\Gamma_{1}$ is no more than $pn$.
\end{proof}

This concludes the proof of the theorem so any $p^{th}$ root of a level $N$ modular function that has a level, in fact, has level $N$ or $pN$.  As $f(z)^{1/p}$ is not level $N$ by assumption, it must be level $pN$.
\end{proof}

The special case we are currently most interested in is when $p=2$, in which case we have the following theorem and its corollary.

\begin{theorem}  If $f(z)\in \mathcal{F}_{N}\setminus \mathcal{F}_{N}^{2}$ is a modular unit and $\sqrt{f(z)}$ has level $M$ for some $M\in\N$ with $N | M$ then, in fact, $\sqrt{f(z)}$ has level $2N$.  
\end{theorem}

\begin{theorem}  If $f(z)\in \mathcal{F}_{N}\setminus \mathcal{F}_{N}^{2}$ is a modular unit with $\sqrt{f(z)} \not\in \mathcal{F}_{N}(\sqrt{j(z)-1728})$ then $\sqrt{f(z)}$ is not level $M$ for any $M\in\N$. 
\end{theorem}
\begin{proof}  By the index, there is a unique quadratic extension between $\mathcal{F}_N$ and $\mathcal{F}_{2N}$.  We observe that $\mathcal{F}_N(\sqrt{j(z) - 1728})$ is such an extension since $j(z)-1728$ has a holomorphic $PSL_2(\Z)$-invariant square root on $\mathcal{H}$.
\end{proof}

The theorem says that if the square root of a modular unit of level $N$ is a modular function on a congruence subgroup then it is level $2N$.  Thus, because the Siegel units $\phi_{u,v}$ for $(u,v) \in \frac{1}{N}\Z$ are level $12N^2$,  it suffices to show the square roots of Siegel functions are not level $24N^2$ in order to conclude that they do not, in fact, have a level at all. For definitions and further discussion of the Siegel units see~\cite{kubertlang:1981}.

\section{Background}\label{background}

\subsection{Modular functions}

Let $\mathcal{H} = \{z\in \C | \operatorname{Im}\{z\}> 0\}$ denote the complex upper half plane; let $\mathcal{H}^{*}=\mathcal{H}\cup\mathbb{P}^{1}(\Q)$ be the extended upper half plane and $\hat{\C}$ the compactified complex plane.  Let $\Gamma$ denote the (inhomogeneous) modular group, or the group of all fractional linear transformations mapping $\mathcal{H}$ to itself.  Then $\Gamma$ is naturally identified with the matrix group  
\[
\Gamma = PSL_{2}(\Z) = SL_{2}(\Z)/\{\pm I\} = \{ \left( \begin{array}{cc} a & b\\ c & d \end{array}\right) | \text{  }a, b, c, d \in \Z,  ad - bc = 1 \} /\{\pm I\},
\]
which is generated by $S=\left( \begin{array}{cc}1 & 1\\ 0 & 1 \end{array}\right)$ and $T=\left( \begin{array}{cc}0 & -1\\ 1 & 0 \end{array}\right)$.    The action of $\Gamma$ on an upper half-plane  variable $z\in \mathcal{H}$ is given via fractional linear transformation:

\[
A \circ z = \frac{az + b}{cz + d}.
\]

A map $f:\mathcal{H}^{*}\rightarrow \hat{\C}$ is called a \emph{modular function (of level one)} if
\begin{enumerate}
\item $f$ is meromorphic on $\mathcal{H}$,
\item $f(A\circ z) = f(z)$ for all $A\in\Gamma$ and $z \in \mathcal{H}^{*}$,
\item there is an $a>0$ so that for $\operatorname{Im}\{z\} > a$, $f(z)$ has an expansion in the local variable at $i \infty$, $q = e^{2\pi i z}$, of the form
\[
f(z) = \sum_{n\geq n_{0}}a_{n}q^{n}, \text{ } n\in\Z, \text{ }a_{n_{0}}\neq 0.
\]
\end{enumerate}
\noindent so $n_{0}$ determines the behavior of $f$ as $z \rightarrow \infty$.  If $n_{0}<0$ then $f(i\infty) = \infty$; if $n_{0}=0$ then $f(i\infty) = a_{0}$; and if $n_{0}>0$ then $f(i\infty) = 0$.  In the last case, we call $f$ a \emph{cusp form}.

Fix a natural number $N>2$.  Let $\Gamma(N) \leq \Gamma$ be the (inhomogeneous) principal congruence subgroup modulo $N$, or the kernel of the reduction mod $N$ map. In other words,

\[
\xymatrix{ 1 \ar[r]  &\Gamma(N) \ar[r]  &\Gamma  \ar[r] & PSL_{2}(\Z/N\Z) \ar[r] &1 }
\]

\noindent is a short exact sequence.  

By convention, we take $\Gamma(1) =\Gamma$.  The upper half-plane modulo the action of $\Gamma$ (written, by abuse of notation, $ \mathcal{H}/  \Gamma$) is a singular surface whose one-point compactification by the image of the point $i \infty$  under the stereographic projection is homeomorphic to the Riemann sphere.  The completed non-singular curve is denoted $X(1)$. Similarly, $\mathcal{H}/\Gamma(N)$ can be compactified by adding finitely many points, the \emph{cusps} of $\Gamma(N)$, or the translates of $i \infty$ under a full set of coset representatives for $PSL_{2}(\Z/N\Z)$ in $\Gamma$.  In this case, the curve is denoted $X(N)$. 

If $H$ is a finite index subgroup in $\Gamma$ the set of \emph{cusps}, or translates of $i \infty$ under a full set of coset representative for $H$ in $\Gamma$, will hereafter be denoted $C(H)$.  A finite index subgroup of $\Gamma$ defined by congruence conditions is called a \emph{congruence subgroup}.  The \emph{conductor} of a congruence subgroup $H$ is the largest $N$ for which $\Gamma(N) \subseteq H$.

\subsection{Modular functions of level $N$}
 A \emph{modular function} for a congruence subgroup $\Gamma(N)$ is a function, $f(z):\mathcal{H}^{*} \rightarrow \hat{\C} $  such that 
\begin{enumerate}
\item $f$ is meromorphic on $\mathcal{H}$,
\item $f(A\circ z)=f(z)$ for all $A \in \Gamma(N)$ and $z \in \mathcal{H}^{*}$,
\item $f(z)$ has an expansion at each of the cusps in the local variable $q = e^{2 \pi i z}$ of the form
\[
f(z) = \sum_{n\geq n_{0}}a_{n}q^{n}, \text{ } n\in\Z, a_{n_{0}}\neq 0.
\]
\end{enumerate}
\noindent If $f$ is modular for $\Gamma(N)$, we say $f$ has \emph{level} $N$.  A modular function of level $N$ descends to a well-defined holomorphic function on $X(N)$. As before, if $n_{0}>0$ for all $\alpha \in C(\Gamma(N))$ then $f(z)$ is called a cusp form for $\Gamma(N)$.


\subsection{The full tower of modular functions $\mathcal{F}$}

The set of modular functions invariant under the full modular group $\Gamma$ is, in fact, a function field of genus one and is generated over  $\C$ by the classical $j$-function,
\[
 j(z)=\frac{1}{q} + 744 + 196884q + 21493760q^{2}+ 864299970q^{3} + O(q^4) .
 \]

 We write $\mathcal{F}_{1}=\Q(j(z))$ and note that $\mathcal{F}_{1}$ is the full field of rational functions on $X(1)$ whose Fourier coefficients are rational. The $j$-function is normalized so that its $q$-expansion at $i\infty$ (which is the only cusp of $\mathcal{H}/\Gamma$) has integral coefficients.  Thus, it is reasonable to define the ring of integers in this field to be $\Z[j]$. 

Furthermore, the set of level $N$ functions together with the $N^{th}$ roots of unity generate a field extension of $\mathcal{F}_{1}$, denoted $\mathcal{F}_{N}$, which is a finite Galois extension of $\mathcal{F}_{1}$ with Galois group $PGL_{2}(\Z/N\Z) \cong \Gamma/\Gamma(N)\times(\Z/N\Z)^{\times}$.  The Galois action is given by writing $PGL_{2}(\Z/N\Z) \cong PSL_{2}(\Z/N\Z) \times (\Z/N\Z)^{\times}$ and letting $PSL_{2}(\Z/N\Z)$ act as usual as fractional linear transformations on $z \in \mathcal{H}$.  A matrix in $PGL_{2}(\Z/N\Z)$ with determinant $d \in (\Z/N\Z)^{\times}$ acts on a (not necessarily primitive) $N^{th}$ root of unity $\zeta$ via $\sigma_{d}:\zeta \mapsto \zeta^{d}$.
\noindent In other words,  if $f$ has Fourier expansion
\[
f(z) = \sum_{n\geq n_{0}}a_{n}q^{n}
\]
\noindent then elements of the form $A = \left( \begin{array}{cc} 1 & 0 \\ 0 & d \end{array}\right) \in PGL_{2}(\Z/N\Z)$ act as follows:

\[
f(A\circ z) = \sum_{n\geq n_{0}}\sigma_{d}(a_{n})q^{n}
\]

\noindent and more general elements $A$ of determinant $d$ act as

\[
f(A \circ z) = \sum_{n\geq n_{0}}\sigma_{d}(a_{n})(A'\circ q)^{n}
\]

\noindent where $A' = \frac{1}{\sqrt{d}}A \in PSL_2(\Z/{N\Z})$.

Taking the integral closure of $\Z[j]$ in $\mathcal{F}_{N}$, we get a ring $R_{N}$, whose units, $U_{N}$, are the \emph{modular units of level N}.  It is not uncommon, however, to extend scalars to $\C$, that is, to study $U_{N}\otimes \C \subseteq R_{N}\otimes \C$. In this setting the set of functions with  multiplicative inverses coincide precisely with the set of function whose divisor of zeros and poles is supported at the cusps of $X(N)$.

Finally, the compositum of the $\mathcal{F}_{N}$ over all $N$ is called the \emph{full tower of modular functions} $\mathcal{F}$. The set of units $U$ in the full tower of modular functions is the direct limit of the $U_{N}$ with respect to the natural inclusion maps.

\nocite{diamondshurman:2005}
\bibliographystyle{alpha}
\bibliography{bibliography.bib}

\end{document}